\documentclass[12pt, leqno]{amsart}
\usepackage{amsmath,amssymb,latexsym,amsthm,mathrsfs}
\usepackage{url}

\usepackage{calc}

\usepackage{graphicx,color,transparent} 

\newcommand\labN[2][12pt]{\makebox(0,0)[cb]{\raisebox{#1}{\mbox{#2}}}}

\newcommand\labNE[2][12pt]{\makebox(0,0)[lb]{\raisebox{#1}{\makebox[\width + #1][r]{#2}}}}
\newcommand\labE[2][12pt]{\makebox(0,0)[lc]{\makebox[\width + #1][r]{#2}}}

\newcommand\labnhS[2][12pt]{\makebox(0,0)[ct]{\raisebox{-#1}{\mbox{#2}}}}

\theoremstyle{plain} 
\newtheorem*{defn*}{Definition}
\newtheorem*{lem*}{Lemma}
\newtheorem{lem}{Lemma}[section]
\newtheorem*{thm*}{Theorem}

\newtheorem*{cor*}{Corollary}
\newtheorem{prop}{Proposition}
\newtheorem*{prop*}{Proposition}
\newtheorem*{claim*}{Claim}

\theoremstyle{definition} 
\newtheorem*{rem*}{Remark}
\newtheorem*{rems*}{Remarks}

\newtheorem*{prob*}{Problem}
\newtheorem*{nota*}{Notation}
\newtheorem*{queORs*}{Question(s)}

\newtheorem*{exa*}{Example}

\newcommand{\abs}[1]{\left\vert#1\right\vert}
\newcommand{\inv}{^{-1}}

\newcommand{\too}{\longrightarrow}

\renewcommand{\iff}{\Leftrightarrow}

\newcommand{\iso}{\cong }
\newcommand{\inp}[1]{ \langle #1 \rangle}

\newcommand{\Span}{\operatorname{span}}

\newcommand{\Stab}{\operatorname{Stab}}
\newcommand{\Fix}{\operatorname{Fix}}
\newcommand{\HFix}{\operatorname{HFix}}

\newcommand{\SL}{\mathrm{SL}}
\newcommand{\GL}{\mathrm{GL}}
\newcommand{\PSL}{\mathrm{PSL}}
\newcommand{\PGL}{\mathrm{PGL}}

\newcommand{\sH}{\mathcal{H}}
\newcommand{\sO}{\mathcal{O}}
\newcommand{\sT}{\mathcal{T}}
\newcommand{\Q}{\mathbb{Q}}
\newcommand{\R}{\mathbb{R}}

\begin{document}

\title[Examples of Fuchsian Groups Sitting in $\mathrm{PSL}_2(\mathbb{Q})$]{Many Examples of Non-cocompact Fuchsian Groups Sitting in $\mathrm{PSL}_2(\mathbb{Q})$}
\author{Mark Norfleet}
\address{Department of Mathematics\\University of Texas--Austin}
\email{mnorfleet@math.utexas.edu}
\urladdr{http://www.ma.utexas.edu/users/mnorfleet/}
%
\begin{abstract}
We construct infinitely many noncommensurable non-cocompact Fuchsian groups $\Delta$ of finite covolume sitting in $\mathrm{PSL}_2(\mathbb{Q})$ so that the set of hyperbolic fixed points of $\Delta$ will contain a given finite collection of elements in the boundary of the hyperbolic plane.
\end{abstract}

\maketitle

\section{Introduction}

Let $\Gamma$ be a Fuchsian group, meaning a discrete subgroup of the group of orientation preserving isometries of $H^2$, the hyperbolic plane. A boundary point of $H^2$ fixed by a parabolic element of a Fuchsian group $\Gamma$ is referred to as a \textit{cusp} of $\Gamma$, and a line fixed by a hyperbolic element is referred to as an \textit{axis} with endpoints called hyperbolic fixed points. For an arbitrary Fuchsian group, determining its set of cusps, hyperbolic fixed points, and/or axes is quite a challenge; for some of the literature addressing this type of problem, see \cite{LongReid02} and its references.

Recall that Fuchsian groups $\Gamma_1$ and $\Gamma_2$ are \textit{commensurable} if $\Gamma_1$ has a subgroup of finite index which is conjugate to a subgroup of finite index in $\Gamma_2$. This work has been motivated by the following question(s):

\begin{queORs*}
If $\Gamma_1$ and $\Gamma_2$ are finite covolume Fuchsian groups with the same set of cusps (or same set of axes), when are they commensurable? 
\end{queORs*}

In \cite{LongReid02}, Long and Reid exhibit four examples of mutually noncommensurable subgroups of $\PSL_2(\Q)$, which are not commensurable with the modular group, but each of them have cusp set exactly $\Q \cup \left\{ \infty \right\}$; they call such groups  \textit{pseudomodular}.  It is still unknown whether or not there are infinitely many pseudomodular groups up to commensurability. Any other possible candidate for pseudomodular groups are (non-arithmetic) discrete subgroups $\Delta \leq \PSL_2(\Q)$, since their cusp set is contained in $\Q \cup \left\{ \infty \right\}$.  For Fuchsian groups, a boundary point cannot both be a cusp and a hyperbolic fixed point. Hence arithmetic and pseudomodular groups cannot have rational hyperbolic fixed points. So if one can exhibit a hyperbolic element of $\Delta$ that has rational fixed points, then $\Delta$'s cusp set is properly contained in $\Q \cup \left\{ \infty \right\}$; thus showing $\Delta$ to be neither arithmetic nor pseudomodular. So Long and Reid asked, how to predict when rational hyperbolic fixed points are present. 

Another motivation arises from a question A. Rapinchuk asked. Are there infinitely many commensurability classes of finite covolume Fuchsian groups sitting in $\PSL_2(\Q)$? Vinberg, answering this question in a preprint \cite{VinbergPP}, has introduced a way to produce infinitely many noncommensurable finite covolume Fuchsian groups in $\mathrm{SL}_2(\mathbb{Q})$.  His examples arise as the even subgroup of a group generated by reflections in the sides of quadrilaterals. To establish that he constructed infinitely many groups up to commensurability, Vinberg uses results (from \cite{Vinberg71}) about the least ring of definition for his examples. 

This paper provides a new solution to Rapinchuk's question, which is different from Vinberg's, and the results also address the presence of rational hyperbolic fixed points.  Namely, we will construct Fuchsian groups sitting in $\PSL_2(\Q)$, all of which will possess a given finite set of rational hyperbolic fixed points.  The main result is

\begin{thm*}
Let $Y$ be a finite set of rational boundary points of the hyperbolic plane. Then there are infinitely many noncommensurable finite covolume Fuchsian groups sitting in $\PSL_2(\Q)$, whose set of hyperbolic fixed points contains $Y$.
\end{thm*}
 
This will be proved in section \ref{proofT}. As a brief outline, let $Y$ be a finite number of boundary points of the hyperbolic plane.  We construct (with considerable freedom) examples of Fuchsian groups $\Gamma$ of signature $(0:2,\dots,2 ;1;0)$ such that the set of hyperbolic fixed points of $\Gamma$ contains $Y$ (see section \ref{constG}); furthermore, when $Y$ is a set of rational boundary points, and by restricting some of the freedom in the construction, one can guarantee $\Gamma \leq \PGL_2(\Q)$.  Then we address when the constructed groups are mutually noncommensurable, which relies on analyzing how they act on different trees (see section \ref{SLtrees}). Specifically, we consider fixed points of the action of $\Gamma$ on Serre's trees of $\SL_2(\Q_p)$ for primes $p \equiv  3 \mod 4$ (see Proposition \ref{noLstab}). This perspective allows us to construct an infinite family of mutually noncommensurable groups in $\PSL_2(\Q)$.

The author is grateful to the University of Texas at Austin for their support during the Summer and Fall of 2012.  I am indebted to my advisor Daniel Allcock for all the conversations we had during my work on this project and in particular for his careful reading of this manuscript. I would also like to thank Alan Reid for helpful discussions about pseudomodular groups.

\section{The Construction}\label{constG}

In this section, we will construct a fundamental domain for a Fuchsian group $\Gamma$ of signature $(0:2,\dots,2 ;1;0)$ so that the set of hyperbolic fixed points of $\Gamma$, notated by $\HFix(\Gamma)$, will contain a given finite set $Y$ of elements in boundary of the hyperbolic plane. Before the construction begins, we introduce some notation and a lemma.

Let $H^2$ be the hyperbolic plane. We will write $\overline{xy}$ for the closure of the geodesic line with end points $x,y$ in the closure of $H^2$. The isometry denoted by $\rho_f $ is rotation by $\pi$ with fixed point $f \in H^2$. Let $v_0$ and $v_\infty$ be distinct elements in $\partial H^2$; then we have an order $\leq$ on $\partial H^2 \setminus \left\{ v_\infty \right\}$ (i.e. the order on the real line in the upper half plane model).  

\begin{lem}\label{chooselem}
Let $ x < v < y < u$ in $\partial H^2 \setminus \left\{ v_\infty \right\}$. For each $f$ in the interior of $\overline{cy}$, where $c = \overline{vu} \cap \overline{xy}$, one constructs $t = \rho_f (u)$ and $w = \rho_f(v)$. Then $ x < v < t < y < w < u$ in $\partial H^2$, $f = \overline{xy} \cap \overline{tu}$, and $f \in \overline{vw}$.
\end{lem}

We omit the proof, but include a figure to help illustrate the lemma.

\begin{figure}[h]
\setlength{\unitlength}{320pt}
\begin{picture}(1,0.35125957)%
    \put(0,0){\includegraphics[width=\unitlength]{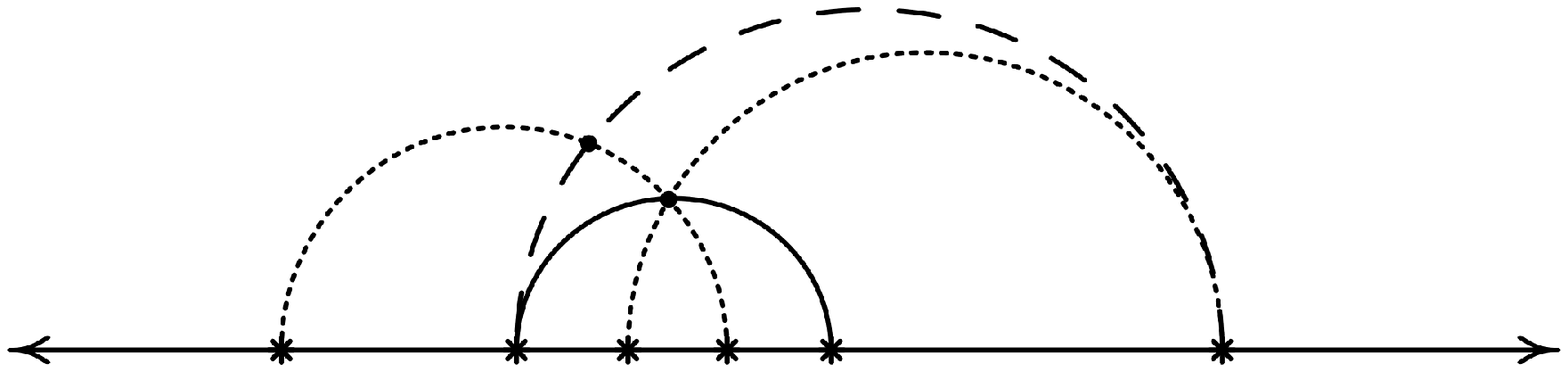}}%
    \put(0.35734077,0.08671665){\labnhS[14pt]{$v$}}
    \put(0.5335716,0.08671659){\labnhS[14pt]{$w$}}
    \put(0.47506919,0.08671662){\labnhS[14pt]{$y$}}
    \put(0.22576178,0.08671661){\labnhS[14pt]{$x$}}
    \put(0.41966765,0.08671663){\labnhS[14pt]{$t$}}
    \put(0.75207754,0.08671659){\labnhS[14pt]{$u$}}
    \put(0.39776601,0.20206234){\labN[6pt]{\makebox[\width + 6pt][l]{$c$}}}
    \put(0.44246167,0.17077751){\labN[0pt]{\makebox[\width + 36pt][r]{$f$}}}
\end{picture}%
\caption{Lemma \ref{chooselem} drawn in upper half plane model.}
\end{figure}

\subsection{The Construction of a Fundamental Domain for $\Gamma$}\label{steps}

As an overview of the construction (Figure 2 illustrates an example), we start with a finite set of points in the boundary of the hyperbolic plane (the set $Y=\left\{ y_i \right\}$). We use Lemma \ref{chooselem} to sequentially construct the vertices and edges (the solid lines in Figure 2) of an ideal convex polygon in $H^2$.  We then show that the ideal convex polygon constructed is a fundamental domain for a discrete group $\Gamma$ generated by isometries that are rotations by $\pi$. Furthermore, $\Gamma$ is guaranteed to possess a set of hyperbolic elements (the dash lines are their axes in Figure 2) whose fixed points contain the initially given set of boundary points.       

Below we will have the notational convention: $f_i \in H^2$, $\rho_i= \rho_{f_i}$ is rotation by $\pi$ with fixed point $f_i$, and $y_i, x_i, v_i \in \partial H^2$. As a slight variation when considering the hyperboloid model of the hyperbolic plane (see subsection \ref{V21U2models}), $f_i, x_i, y_i, v_i$ will be vectors in Minkowski space $\R^{2,1}$.

We begin the construction; let $Y$ be a finite set of $n-1$ points in the boundary of the hyperbolic plane, and let $Y= \left\{ y_i \right\}$ so that $v_0 < y_1 < \dots < y_{n-1} \ne v_\infty$ in the boundary of the hyperbolic plane.

\vspace{12pt}

\noindent\hfill\parbox[c]{.95\textwidth}{\textbf{$1^{\text{st}}$ Step:} Choose $x_1$ such that $v_0 < x_1 < y_1$ in $\partial H^2$, and then choose $f_1 \in \overline{x_1 y_1}$.  Define $v_1 = \rho_{1}(v_0)$; note $v_0 < x_1 < v_1 < y_1$ in $\partial H^2$, and $f_1 \in \overline{v_0 v_1}$.}

\vspace{12pt}

\noindent\hfill\begin{minipage}[c]{\textwidth}
\noindent When $n>2$, let $i \in \left\{ 2, \dots ,n-1 \right\}$.

\vspace{6pt}

\noindent\hfill\parbox[c]{.95\textwidth}{\textbf{$i^{\text{th}}$ Step:} Let $x_{i-1} < v_{i-1} < y_{i-1}< y_{i}$ in $\partial H^2$. By Lemma~\ref{chooselem}, one can choose a $f_i \in \overline{x_{i-1}  y_{i-1}}$ and construct $x_{i}=\rho_{i}(y_i)$ and $v_i= \rho_{i}(v_{i-1})$, so that $v_{i-1} < x_{i} < y_{i-1} < v_{i} < y_{i} $ and $f_i \in \overline{v_{i-1} v_i}$.}
\end{minipage}

\vspace{12pt}

\noindent\hfill\parbox[c]{.95\textwidth}{\textbf{$n^{\text{th}}$ Step:} Let $x_{n-1} < v_{n-1} < y_{n-1} \ne v_\infty$ in $\partial H^2$. Now construct $f_{n} = \overline{x_{n-1} y_{n-1}} \cap \overline{v_{n-1} v_\infty}$, and define $v_{n} = \rho_{n}(v_{n-1})$.}

\vspace{12pt}

\noindent\hfill\parbox[c]{.95\textwidth}{\textbf{Last Step:} Given $\rho_{n}  \cdots  \rho_{1} (v_0) = v_n$, construct $f_0 \in \overline{v_0 v_n}$ and $\rho_{0}$ so that $\rho_{n}  \cdots  \rho_{1} \rho_0 $ is parabolic fixing $v_n$.}

\vspace{12pt}

\begin{rem*}
When this construction is done with vectors in Minkowski space $\R^{2,1}$, $v_n$ and $v_\infty$ are linearly dependent light-like vectors, and in the last step one can see for $f_0 \in \Span\left\{ v_n + v_0  \right\}$, the element $\rho_{0}$ as a Lorentz transformation maps $v_n$ to $v_0$ (as vectors in $\R^{2,1}$), which shows $\rho_{n}  \cdots  \rho_{1} \rho_0 $ is parabolic fixing $v_n$. 
\end{rem*}

\begin{figure}[h]\label{FDGamma}
\setlength{\unitlength}{320pt}
  \begin{picture}(1,0.51687422)%
    \put(0,0){\includegraphics[width=\unitlength]{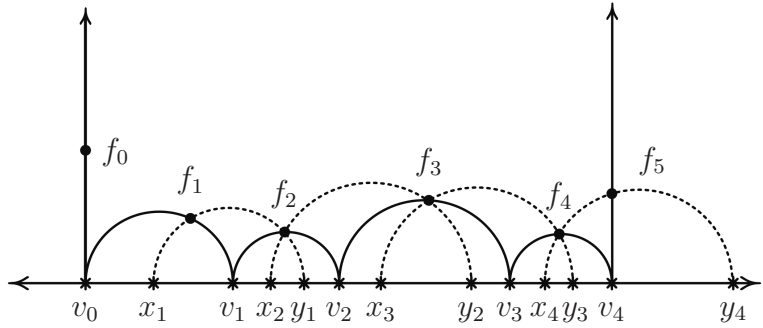}}%
    \put(0.14720675,0.12782516){\labnhS{$v_0$}}
    \put(0.32135845,0.12782516){\labnhS{%
$v_1$%
}}
    \put(0.44700894,0.12782537){\labnhS{%
$v_2$%
}}
    \put(0.64865714,0.12782533){\labnhS{%
$v_3$%
}}
    \put(0.76928582,0.12782507){\labnhS{%
$v_4$%
}}
    \put(0.76928584,0.23365269){\labNE[10pt]{%
$f_5 $%
}}
    \put(0.70673114,0.18585148){\labN{%
$f_4 $%
}}
    \put(0.5529505,0.22573806){\labN{%
$f_3 $%
}}
    \put(0.38184689,0.18825234){\labN{%
$f_2 $%
}}
    \put(0.27118875,0.20435612){\labN{%
$f_1 $%
}}
    \put(0.14720677,0.28485572){\labE[6pt]{%
$f_0 $%
}}
    \put(0.227606,0.12782513){\labnhS{%
$x_1$%
}}
    \put(0.36598355,0.12782486){\labnhS{%
$x_2$%
}}
    \put(0.49626757,0.12782462){\labnhS{%
$x_3$%
}}
    \put(0.69021535,0.12782449){\labnhS{%
$x_4$%
}}
    \put(0.40552013,0.12782512){\labnhS{%
$y_1$%
}}
    \put(0.60320266,0.12782538){\labnhS{%
$y_2$%
}}
    \put(0.7231411,0.12782458){\labnhS{%
$\, y_3$%
}}
    \put(0.91299813,0.12851216){\labnhS{%
$y_4$%
}}
  \end{picture}%
\caption{The solid lines bound a fundamental domain for $\Gamma$ in the upper half plane model with $Y=\left\{ y_1, \dots , y_4 \right\}$.}
\end{figure}

We have an ideal $n+1$ sided convex polygon, $P$, with vertices $\left\{ v_0, v_1, \dots, v_{n-1}, v_n=v_\infty \right\}$; furthermore, $f_i$ is on the edge $\overline{v_{i-1} v_i}$, and $f_0$ is on the edge $\overline{v_n v_0}$ (see Figure 2).  Since $\rho_i$ is rotation by $\pi$ with fixed point $f_i \in H^2$, $\rho_i$ maps $\overline{v_{i-1} v_i}$ to itself (likewise, $\rho_0$ maps $\overline{v_n v_0}$ to itself); that is, $\rho_i$ maps the directed edge $\overline{ f_i  v_{i-1}}$ to the directed edge $\overline{ f_i  v_i }$, and $\rho_0$ maps the directed edge $\overline{ f_0  v_{0}}$ to the directed edge $\overline{ f_0  v_n }$. By Poincar\'e's Theorem (see section \S9.8 in \cite{Beardon83}), the group $\Gamma$ generated by $\left\{  \rho_{1}, \dots ,  \rho_{n}, \rho_{0} \right\}$ is discrete, and $P$ is a fundamental domain for $\Gamma$. In the last step, we made $\rho_{n}  \dots  \rho_{1} \rho_0 $ parabolic fixing $v_n$; thus $H^2/\Gamma$ is a complete finite area once punctured $2$-sphere with $n+1$ cone points of order $2$.

For $1\leq i < n$, the element $\rho_i \rho_{i+1}$ is hyperbolic with axis $\overline{x_i y_i}$, since $f_{i}$ and $f_{i+1}$ both lie on the geodesic line $\overline{x_i y_i}$ (by construction); therefore, $y_i$ is a hyperbolic fixed point for $\rho_i \rho_{i+1} \in \Gamma$.  

By Lemma~\ref{chooselem}, one sees that there are infinitely many choices for each $f_i$ (for $1 \leq i < n$), producing infinitely many such Fuchsian groups $\Gamma$; establishing the following:

\begin{prop}
Let $Y$ be a finite set of $n-1$ points in $\partial H^2$. Then there are infinitely many Fuchsian groups $\Gamma$ of finite covolume of signature $(0: \underbrace{2,\dots,2}_{n+1}\; ;1;0)$ such that $Y \subset \HFix(\Gamma)$.   
\end{prop}

\subsection{$\Gamma$ in $\mathrm{O}^+(2,1)$ and $\PGL_2(\R)$.} \label{V21U2models}

Let $V$ be a dimension $3$ real vector space with a nondegenerate quadratic form $ \inp{\bullet,\bullet}$ of signature $(2,1)$. Choose a basis $(e_1, e_2 , e_0)$ with $\inp{e_i,e_j}=0$ if $i\ne j$, $\inp{e_i,e_i}=1$ if $i\geq1$, and $\inp{e_0,e_0}=-1$. Such a basis is called a Lorentz orthonormal basis, and $V$ is denoted as $\R^{2,1}$ when such a basis is fixed; $\R^{2,1}$ is called Minkowski space.  We will notate $L=\left\{ v: \inp{v,v}=0 \right\}$ (the set of light-like vectors) and $T=\left\{ v: \inp{v,v}<0 \right\}$ (the set of time-like vectors).  Let $L^+$ and $T^+ $ be the sets of vectors with positive $e_0$-coordinate in $L$ and $T$, respectively.   We let $\mathrm{O}^+(2,1)$ be the group of linear transformations of $V$ that preserve the quadratic form and upper sheet of the hyperboloid $\sH = \left\{ v : \inp{v,v}=-1 \right\} \cap T^+$. With a Lorentz orthonormal basis $\left\{ e_1, e_2, e_0   \right\}$, let the $\Q$-linear combination of $\left\{ e_1, e_2, e_0 \right\}$ be denoted by $\Q^{2,1}$; furthermore, let $L_\Q^+ = L^+ \cap \Q^{2,1}$ and $T_\Q^+ = T^+ \cap \Q^{2,1}$.

For every $v \in T^+$, there is a $2$ by $2$ real symmetric matrix whose determinant is $-\inp{v,v}$; namely,
\[
v \mapsto \left( \begin{matrix}
\inp{v,e_2 + e_0} & \inp{v,e_1} \\
\inp{v,e_1} & \inp{v, -e_2 + e_0}
\end{matrix} \right).
\]
We have that $\GL_2(\R)$ acts on $2$ by $2$ real symmetric matrices by similarity; that is, $\Sigma \mapsto M^{\text{t}} \Sigma M$, where $M \in \GL_2(\R)$ and $\Sigma$ is a $2$ by $2$ real symmetric matrix. We use this to relate the hyperboloid and upper half plane models.

\begin{exa*}
Consider the isometry $\rho_f$, rotation by $\pi$ with fixed point $f \in H^2$. 

In the hyperboloid model, $\rho_f$ corresponds to an element in $\mathrm{O}^+(2,1)$; let $f \in T^+$ be fixed by $\rho_f \in \mathrm{O}^+(2,1)$. Let $\Sigma_f$ be the $2$ by $2$ real symmetric matrix associated to~$f$,
\[
\Sigma_f = 
\left( \begin{matrix}
\inp{f,e_2 + e_0} & \inp{f,e_1} \\
\inp{f,e_1} & \inp{f, -e_2 + e_0}
\end{matrix} \right).
\]
In the upper half plane model of $H^2$, say $a+b i$ is fixed by $\rho_f$ as a matrix in $\SL_2(\R)$; that is, 
\[
\rho_f = 
\left( \begin{matrix}
b & a \\
0 & 1
\end{matrix} \right)
\left( \begin{matrix}
0 & -1 \\
1 & 0
\end{matrix} \right)
\left( \begin{matrix}
b & a \\
0 & 1
\end{matrix} \right) \inv =
\frac{1}{b} 
\left( \begin{matrix}
a & -(b^2 + a^2)\\
1 & -a
\end{matrix} \right). \tag{$*$}
\]
We have that $\rho_f$, as the matrix from ($*$), acts by similarity on $2$ by $2$ real symmetric matrices and fixes $\Sigma_f$ when 
\[
b = \frac{\sqrt{\abs{\inp{f,f}}}}{ \inp{ f , e_2 + e_0}} \quad \text{ and } \quad a = - \frac{ \inp{f , e_1}}{ \inp{f, e_2+e_0} }.
\]
Note: in $\PGL_2(\R)$, $\rho_f$ can be represented by a matrix with entries in $\Q$ when $\inp{f,f}$, $\inp{f , e_i}$ are all in $\Q$. 
\end{exa*}

\begin{prop}\label{DinpslQ}
Let $v_0 < y_1 < \cdots < y_{n-1} \ne v_\infty$ in $L^+_\Q $. Then there are infinitely many non-cocompact Fuchsian groups $\Delta$ of finite covolume sitting in $\PSL_2(\Q)$ such that $\left\{ y_1, \dots, y_{n-1} \right\} \subset \HFix(\Delta)$.
\end{prop}

\begin{proof}
We follow the construction of $\Gamma$ in Minkowski space. In the $i^{\text{th}}$ step (for $1 \leq i < n$) of subsection \ref{steps}, we additionally require the choice of $x_1$ and the $f_i$, as vectors in Minkowski space, to lie in $\Q^{(2,1)}$. Furthermore, $f_n$ and $f_0$ will be in $\Q^{(2,1)}$, since all $x_i,v_i,y_i$ will be in $\Q^{(2,1)}$. Thus $\Gamma$ will sit in $\PGL_2(\Q)$. Even with these additional requirements in the construction, there are still infinitely many choices for each $f_i$ (for $1 \leq i < n$), producing infinitely many such Fuchsian groups $\Gamma$.

For each $\Gamma$, $\left\{ y_i \right\} \subset \HFix(\Gamma)$, and let $\Delta$ be the kernel of $\Gamma \too \PGL_2(\Q)/ \PSL_2(\Q)$, which is of finite index in $\Gamma$; therefore, it follows $\left\{ y_i \right\} \subset \HFix(\Delta)$.
\end{proof}

\section{Acting on The Tree of $\SL_2(\Q_p)$}\label{SLtrees}

As in Serre's book \cite{Serre80}, let $K$ denotes a field with a discrete valuation $v$; recall that $v$ is a homomorphism of $K^\times$ onto $\mathbb{Z}$, and $\sO_v$ denotes the valuation ring of $K$, i.e. the set of $x \in K$ such that $v(x)\geq 0$ or $x=0$. Fix an element $\pi \in K$ with $v(\pi)=1$, the uniformizer. If $K = \Q$, then most $v$ subscripts are replaced with the letter $p$ for the $p$-adic valuation $v_p$.

Let $V$ be a vector space of dimension $2$ over $K$. A lattice in $V$ is any finitely generated $\sO_v$-submodule of $V$ which generates the $K$-vector space $V$; such a module is free of rank 2. The group $K^\times$ acts on the set of lattices; we call the orbit of a lattice $L$ under this action its class (at times notated $[L]=\Lambda$), and two lattices belonging to the same class are called equivalent. The set of lattice classes is denoted by $\sT_v$, which is made into a combinatorial graph with edges between $\Lambda_0$ and $\Lambda_1$ when $[L_i]= \Lambda_i$ such that $L_0 \leq L_1$ and $L_1/L_0 \iso \sO_v/\pi \sO_v$. Serre proved that $\sT_v$ is a tree. 

\begin{lem} \label{actlem1}
Let $\Gamma$ be generated by a finite number of $\rho_j$, where 
$
\rho_j = C_j 
\left( \begin{smallmatrix}
0 & -1 \\
1 & 0
\end{smallmatrix} \right)
C_j\inv
$
with $C_j \in \GL_2(K)$. 

Then the following are equivalent statements about the action of $\Gamma$ on the tree $\sT_v$:
\begin{enumerate}
\item $\Gamma \leq \Stab(\Lambda)$ for some $\Lambda \in \sT_v$;
\item 
$ \bigcap C_j \Fix_{\sT_v} \left(  
\left( \begin{smallmatrix}
0 & -1 \\
1 & 0
\end{smallmatrix} \right)  \right) \ne \emptyset $;
\item For each pair $(m,k)$, $\Fix_{\sT_v} (\rho_m \rho_k)  \ne \emptyset$.
\end{enumerate}
\end{lem}
\begin{proof}
($1$) $\iff$ ($2$): follows from the equality 
\[
\Fix_{\sT_v} \left(%
C_j 
\left( \begin{smallmatrix}
0 & -1 \\
1 & 0
\end{smallmatrix} \right)
C_j\inv \right)
=
C_j \Fix_{\sT_v} \left(  
\left( \begin{smallmatrix}
0 & -1 \\
1 & 0
\end{smallmatrix} \right)  \right).
\]

($3$) $\iff$ ($1$): see \cite{Serre80}, I\S 6.5.
\end{proof}

\begin{lem}\label{vablem}
Let $\Gamma$ be generated by a finite number of $\rho_j$, where 
$
\rho_j = C_j 
\left( \begin{smallmatrix}
0 & -1 \\
1 & 0
\end{smallmatrix} \right)
C_j\inv
$
with $C_j = \left( \begin{smallmatrix}
b_j & a_j \\
0 & 1
\end{smallmatrix} \right) \in \GL_2(K)$, as in ($*$). 

When $-1$ is not a square in $K$, the following are equivalent:
\begin{enumerate}
\item $\Gamma \leq \Stab(\Lambda)$ for some $\Lambda \in \sT_v$;
\item For each pair $(m,k)$, 
\[
C_k\inv C_m \in \GL_2(\sO_v)\text{;}
\]
\item For each pair $(m,k)$,
\[
v(a_m -a_k) \geq v(b_m)= v(b_k).
\]
\end{enumerate}
\end{lem}
\begin{proof}
When $-1$ is not a square in $K$, we have that $ \Fix_{\sT_v} \left(  
\left( \begin{smallmatrix}
0 & -1 \\
1 & 0
\end{smallmatrix} \right)  \right) = \left\{ [\sO_v^2] \right\}$, where $\sO_v^2$ is the standard lattice.

Let $\rho_m$ and $\rho_k$ be two generators of $\Gamma$.  Then $\Fix_{\sT_v} (\rho_m \rho_k)  \ne \emptyset$ if and only if 
\[
 C_m \Fix_{\sT_v} \left(  
\left( \begin{smallmatrix}
0 & -1 \\
1 & 0
\end{smallmatrix} \right)  \right) 
\bigcap
 C_k \Fix_{\sT_v} \left(  
\left( \begin{smallmatrix}
0 & -1 \\
1 & 0
\end{smallmatrix} \right)  \right) \ne \emptyset, 
\]
and that holds only when $C_k\inv C_m \in \GL_2(\sO_v)$. By Lemma~\ref{actlem1}, (1) and (2) are equivalent.

To complete the proof note that $ C_k\inv C_m = \left( \begin{smallmatrix}
\frac{b_m}{b_k} & \frac{a_m-a_k}{b_k} \\
0 & 1
\end{smallmatrix} \right) \in \GL_2(\sO_v)$ if and only if 
\[
v\left( \frac{b_m}{b_k} \right) =0 \qquad \text{ and } \qquad v\left( \frac{a_m-a_k}{b_k} \right) \geq 0.
\]~\end{proof}

\begin{prop}\label{noLstab}
Let $v_0 < y_1 < y_2 < \dots < y_{n-1} \ne v_\infty$ in $L^+_\Q$ ($n > 2$), and let a prime $p \equiv  3 \mod 4$. Then there are non-cocompact Fuchsian groups $\Delta$ of finite covolume sitting in $\PSL_2(\Q)$ with $\left\{ y_1, \dots, y_{n-1} \right\} \subset \HFix(\Delta)$, each of which stabilize no vertex in the tree $\sT_{p}$.
\end{prop}

\begin{proof}
Consider the construction of $\Gamma$ in subsection \ref{steps} in Minkowski space; below we will describe additional requirements for choosing $x_1$ and the $f_i$.   

In the $1^{\text{st}}$ step, choose $x_1 \in L^+_\Q $ so $v_0 < x_1 < y_1$.  When choosing $f_1$, additionally require that $f_1 \in \Span_\Q \left\{ x_1,y_1 \right\}$ and $\abs{\inp{f_1,f_1}}$ is in the rational square class $p(\Q^\times)^2=\left\{ p \alpha^2: \alpha \in \Q \right\}$, which is possible because $\Span_\Q \left\{ x_1,y_1 \right\}$ is isotropic. In the $i^{\text{th}}$ step (for $1 < i < n$), one specifies a rational square class, say $n_i (\Q^\times)^2 \in \Q^\times / (\Q^\times)^2$ ($n_i$ a square free integer) such that $p \not | n_i$. When choosing $f_i$, additionally require that $f_i \in \Span_\Q \left\{ x_{i-1},y_{i-1} \right\}$ and $\abs{\inp{f_i,f_i}} \in n_i(\Q^\times)^2$, which is also possible because $\Span_\Q \left\{ x_{i-1} , y_{i-1} \right\}$ is isotropic. 

Now note (as in the example in subsection \ref{V21U2models}) that $\rho_{f_i}$ as an element of $\PGL_2(\Q)$ is given by the matrix 
\[
\rho_{f_i}  = 
\left( \begin{matrix}
b_i & a_i \\
0 & 1
\end{matrix} \right)
\left( \begin{matrix}
0 & -1 \\
1 & 0
\end{matrix} \right)
\left( \begin{matrix}
b_i & a_i \\
0 & 1
\end{matrix} \right) \inv,
\]
where $b_i = \frac{\sqrt{\abs{\inp{f_i,f_i}}}}{ \inp{ f_i , e_2 + e_0}}$ and $a_i = - \frac{ \inp{f_i , e_1}}{ \inp{f_i, e_2+e_0} }$. 

Since $b_1$ has a factor $\sqrt{p}$, and each $b_i$ ($1 < i < n$) does not,
\[
v_{p} (b_i) \ne  v_{p}(b_1).
\]
Moreover, $-1$ is not a square in $\Q_{p}$ (since $p \equiv  3 \mod 4$). By Lemma~\ref{vablem}, $\Gamma$ stabilizes no vertex in $\sT_{p}$, and by construction $\left\{ y_i \right\} \subset \HFix(\Gamma)$. Now let $\Delta$ be the kernel of $\Gamma \too \PGL_2(\Q)/ \PSL_2(\Q)$, which is of finite index in $\Gamma$; then $\Delta$ also stabilizes no vertex in $\sT_{p}$ and $\left\{ y_i \right\} \subset \HFix(\Delta)$.
\end{proof}

\begin{rem*}
For each $\Delta$ constructed in the proof of Proposition \ref{DinpslQ}, there is an integer $m$ such that $\Delta$ stabilizes a vertex of $\sT_{q}$ for all primes $q> m$. To see this, choose $m$ large enough so that $m$ is greater than all the denominators of the entries of a matrix representing $\rho_{f_i}$, for each $i$, as an element of $\PGL_2(\Q)$
\end{rem*}

\section{Proof of the Theorem}\label{proofT}

\begin{thm*} \label{mainthm}
Let $Y$ be a finite set of rational points in the boundary of the hyperbolic plane.  Then there are infinitely many noncommensurable non-cocompact Fuchsian groups $\Delta$ of finite covolume sitting in $\mathrm{PSL}_2(\mathbb{Q})$ so that $Y \subset \HFix(\Delta)$.
\end{thm*}

\begin{proof}
We can let $Y$ be a finite set of two or more rational points in the boundary of the hyperbolic plane (just add points if fewer than $2$ are given).  Let $Y= \left\{ y_i \right\}$, so that $v_0 < y_1 < y_2 < \dots < y_{n-1} \ne v_\infty$ in $L^+_\Q$ ($n > 2$). Now let the family $\left\{ \Delta \right\}$ be the set of non-cocompact Fuchsian groups of finite covolume sitting in $\PSL_2(\Q)$ such that $\left\{ y_1, \dots, y_{n-1} \right\} \subset \HFix(\Delta)$, which are constructed in the proof of Proposition \ref{DinpslQ}. 

Assume for the purpose of contradiction that there is a finite number $k$ of commensurability classes in family $\left\{ \Delta \right\}$. Let $\left\{ \Delta_1, \dots, \Delta_k \right\}$ be distinct representatives from the $k$ commensurability classes.

From the remark just after Proposition \ref{noLstab}, there is an integer $m$ such that each $\Delta_j $ ($1 \leq j \leq k$) stabilizes a vertex in $\sT_{q}$ for all $q> m$.   By Dirichlet's theorem on arithmetic progressions, we can choose a prime $p > m$ and $p \equiv  3 \mod 4$.  By Proposition \ref{noLstab}, there is $\Delta_{k+1} \in \left\{ \Delta \right\}$ with $\left\{ y_i \right\} \subset \HFix(\Delta_{k+1})$ and so that $\Delta_{k+1}$ does not stabilize any vertex of $\sT_{p}$.  Therefore, each $\Delta_j$ ($1 \leq j \leq k$) stabilizes a vertex in $\sT_{p}$ but $\Delta_{k+1}$ does not. For subgroups of $\PSL_2(\Q)$, the presence or absence of fixed points in $\sT_{p}$ descends to finite index subgroups, and is invariant under conjugation. Thus $\Delta_{k+1}$ is not commensurable with any of the $\Delta_{j}$ ($1 \leq j \leq k$), which contradicts the assumption there are a finite number of commensurability classes in family $\left\{ \Delta \right\}$.
\end{proof}

\begin{rem*}
By using Proposition \ref{noLstab} and the remark just after it, one can inductively constructs an infinite family where the members lie in different commensurability classes. 
\end{rem*}

As mentioned in the introduction, a boundary point cannot both be a cusp and a hyperbolic fixed point, for Fuchsian groups; thus a direct corollary of the theorem is

\begin{cor*}
Let $Y$ be finite set of rationals. Then there are infinitely many noncommensurable non-cocompact Fuchsian groups of finite covolume sitting in $\mathrm{PSL}_2(\mathbb{Q})$ whose cusp set is properly contained in $\left( \Q \setminus Y  \right) \cup \left\{ \infty \right\}$.
\end{cor*}


\end{document}